\documentclass[12pt]{article}
\usepackage{amssymb,amsmath}
\input{amssym.def}\input{amssym}
\usepackage{graphicx, tikz}

\def\subjclass#1{\par\medskip
\noindent\textbf{Mathematics Subject Classification (2010):} #1}
\def\keywords#1{\par\medskip
\noindent\textbf{Keywords.} #1}

\newcommand{\R}{{\mathbb R}}
\newcommand{\T}{{\mathbb T}}

\renewcommand{\S}{{\mathbb S}}

\newcommand\cB{{\mathcal B}}
\newcommand\cC{{\mathcal C}}

\newcommand\cO{{\mathcal O}}
\newcommand\cM{{\mathcal M}}

\newcommand\bT{{\mathbb T}}

\newtheorem{theorem}{Theorem}[section]
\newtheorem{prop}{Proposition}[section]
\newtheorem{corollary}{Corollary}[section]
\newtheorem{lemma}{Lemma}[section]

\newtheorem{definition}{Definition}[section]
\newenvironment{proof}{\noindent {\bf Proof.}}{ \hfill $\Box$\\ }
\newenvironment{proofof}[1]{\noindent {\bf Proof of #1.}}{ \hfill $\Box$\\ }

\newcommand\eps{\epsilon}

\newcommand{\rf}{r}
\def\eps{\varepsilon}

\def\Leb{\mbox{Leb}}

\title{On volume preserving almost Anosov flows}
\author{Henk Bruin
\thanks{Faculty of Mathematics, University of Vienna, 
Oskar Morgensternplatz 1, 1090 Vienna, Austria; {\it henk.bruin@univie.ac.at}.
}
}
\date{\today}

\begin{document}

\maketitle

\abstract{
The purpose of this paper is to establish limit laws
for volume preserving almost Anosov flows on $3$-three manifolds
having a neutral periodic of cubic saddle type. In the process, we derive estimates 
for the Dulac maps for cubic neutral saddles in planar vector fields.
}



\section{Introduction}

A flow $\phi^t: \cM \times \R \to \cM$ on a (in our setting $3$-dimensional) compact differentiable manifold 
$\cM$ is called {\em Anosov} if its tangent bundle has a continuous flow-invariant mutually transversal
splitting into a neutral flow direction $E^c$, a hyperbolically stable direction
$E^s$ and a hyperbolically unstable direction  $E^u$.
The uniform hyperbolicity of such flows enables one to show various ergodic and statistical properties,
such as ergodicity (if the flow is topologically mixing)
and the Central Limit Theorem (CLT) for H\"older continuous observables.

We obtain an almost Anosov flow (see Definition~\ref{def:aaf} below) by inserting a neutral 
orbit $\Gamma \simeq \{ (0,0)\} \times \S^1$
near which the flow has the following form in local Euclidean coordinates:

\begin{equation}\label{eq:polyvf}
 \begin{pmatrix}
  \dot x \\ \dot y\\ \dot z
 \end{pmatrix}
 = 
 X \begin{pmatrix} x \\ y \\ z \end{pmatrix} 
 =
 \begin{pmatrix} x(a_0 x^2 + a_1 xy + a_2 y^2)) \\ 
 -y(b_0 x^2 + b_1 xy + b_2 y^2)) \\ 1 + w(x,y) \end{pmatrix}  + \cO(4)
\end{equation}
where $\cO(4)$ indicates terms of order four and higher, and the parameters satisfy 
\begin{equation}\label{eq:para}
 a_1, b_1 \in \R,\, a_0, a_2, b_0, b_2 \geq 0 \text{ with } \Delta := a_2b_0 - a_0b_2 \neq 0
\text{ and } c_1^2 < 4c_0c_2
\end{equation}
for $c_i := a_i+b_i,\ i = 0,1,2$. 
That is, the vector field is cubic in the transversal direction to $\Gamma$, but this is 
the only source of non-hyperbolicity.
Finally, $w$ is a linear combination of homogeneous functions in $x$ and $y$, vanishing at $(0,0)$.
Thus period of $\Gamma$ is its length.

The original motivation to study such system was to have a class of natural examples 
of non-uniformly hyperbolic invertible maps
(think of the Poincar\'e map on a section $\Sigma \subset \R^2 \times \{ 0 \}$ or the time-$1$ map
$f_{hor} = \phi_{hor}^1$ for the horizontal flow where only the $x$ and $y$ coordinates are taken into account:
\begin{equation}\label{eq:horivf}
\begin{pmatrix} \dot x \\ \dot y \end{pmatrix} 
= X_{hor}\begin{pmatrix} x \\ y \end{pmatrix} =
\begin{pmatrix} x(a_0x^2+a_1xy+a_2y^2) \\
 -y(b_0x^2+b_1xy+b_2 y^2)
\end{pmatrix} + \cO(4),
\end{equation}
with the restrictions \eqref{eq:para},
as natural examples where operator renewal theory can be applied to get 
precise statistical laws for the flow.
Initially, in \cite{BT17} for the parameter range $\beta_2 := \frac{a_2+b_2}{2b_2} \leq 1$ 
where $f_{hor}$ preserves an infinite Sinai-Bowen-Ruelle (SRB)  measure, we gave mixing rates for $C^1$ observables.
Later \cite{BTT18}, and more relevant to this paper, in the parameter range 
$\beta_2 > 1$ where the flow $\phi^t$ preserves a finite SRB-measure, we established limit laws (Stable Laws and the CLT with standard or 
non-standard scaling, depending on whether $\beta_2 \in (1,2)$, $\beta = 2$ or $\beta_2 > 2$).

All these results were obtained in the absence of {\em mixed terms}, i.e., $a_1=b_1=0$ in \eqref{eq:polyvf}.
This is of course not a natural assumption, and to our knowledge there is no change of coordinates that allows 
one to remove the mixed terms. In fact, if $c_1^2 > 4c_1c_2$, then the behaviour near the saddle is locally 
non-conjugate to the behaviour when $c_1^2 < 4c_1c_2$.

The purpose of this paper is to perform the analysis when mixed terms are present.
The crux of the analysis is the existence of a local first integral (and its explicit form when  
$\cO(4)$-terms are absent in
\eqref{eq:horivf2}), which allows us to reduce the ODE to dimension one.
We will show in Lemma~\ref{lem:mix_lf}
that the first integral $L$ can be found if
\begin{equation}\label{eq:condL}
\frac{b_1}{a_1} 
= \frac{b_2a_0+a_2+2b_0b_2}{b_2a_0+a_2+2a_0a_2}.
\end{equation}
This is a co-dimension one condition in parameter space. However, if we also stipulate that the flow $\phi^t$ 
is volume preserving,
we must assume that $\mbox{div}\, X = 0$ in \eqref{eq:polyvf}, which is equivalent to $\mbox{div}\, \cO(4) = 0$ together with
\begin{equation}\label{eq:divergencefree}
3a_0 = b_0, a_2 = 3b_2, a_1 = b_1.
\end{equation}
From these conditions, \eqref{eq:condL} follows automatically, and therefore
\eqref{eq:polyvf} describes a generic volume preserving almost Anosov flow with a single neutral 
periodic orbit of cubic saddle type. We present the results on limit laws in the volume preserving setting, 
see Corollary~\ref{cor:LimitLaws}.

Central to the proof is the analysis of the Dulac map near the neutral equilibrium of \eqref{eq:horivf}.
This means that we take an incoming and an outcoming transversal to the flow, in our case an unstable
leaf $W^u(0,\eta)$, $\eta \in [\eta_0, \eta_1]$, and a stable leaf $W^s(\zeta_0,0)$, see Figure~\ref{fig:vf},
and the {\em Dulac map} $D:W^u(0,\eta) \to W^s(\zeta_0,0)$ assigns the first intersection 
$\phi_{hor}^T(\eta, \xi_0)$ of the integral curve through $(\eta, \xi_0)$ with the outgoing 
transversal $W^s(\zeta_0,0)$, and the corresponding flow-time is denoted as $T$.
The main technical result of this paper are precise estimates of the Dulac map when
\eqref{eq:horivf} contains mixed terms, but using the assumption \eqref{eq:condL}.

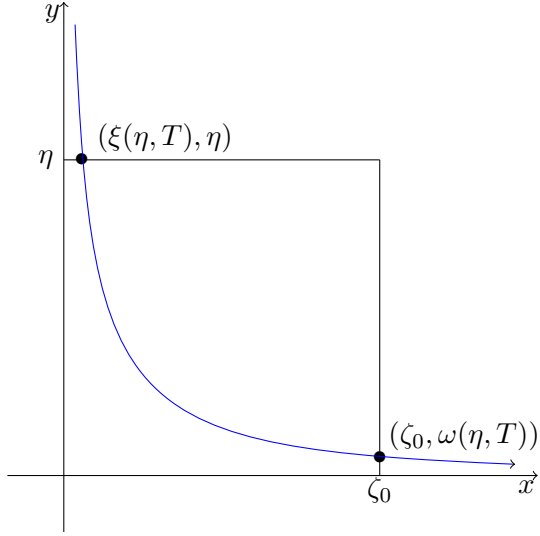
\begin{figure}[ht]
\begin{center}
\begin{tikzpicture}[scale=1.5]
\node at (4.1,-0.1) {\small $x$};
\node at (-0.1,4.1) {\small $y$};
\draw[->] (-0.5,0)--(4.2,0);
\draw[->] (0,-0.5)--(0,4.2);
\draw[-] (0,2.8)--(2.8,2.8); \node at (0.9,3) {\small $(\xi(\eta, T),\eta)$};
\node at (0.16,2.8) {\small $\bullet$}; 
\draw[-] (2.8,0)--(2.8,2.8); \node at (3.55,0.35) {\small $(\zeta_0,\omega(\eta, T))$};
\node at (2.8, 0.16) {\small $\bullet$}; 
\node at (-0.16,2.8) {\small $\eta$};
\node at (2.8, -0.13) {\small $\zeta_0$};
\draw[->, draw=blue] (0.1,4) .. controls (0.25,0.5) and (0.5,0.25) .. (4,0.1);
\end{tikzpicture}
\caption{The Dulac map $D:(\xi(\eta_0, T),\eta_0) \mapsto (\zeta_0,\omega(\eta_0, T))$ with Dulac time $T$.}
\label{fig:vf}
\end{center}
\end{figure} 

Dulac \cite{Dul} introduced his map as an ingredient to prove that polynomial vector fields 
in the plane have at most finitely many limit cycles, thus making a major
contribution to the solution of Hilbert's 16th problem.
\'Ecalle \cite{Ec92} and 
Il'yashenko \cite{I91} independently corrected some weak parts in Dulac's arguments, 
see also the summary in Roussarie's book \cite[Chapter 3 and Section 3.3]{R98}.
Hilbert's problem reduces to Dulac's problem, namely
that polycycles (i.e., heteroclinic saddle connections) 
cannot accumulated upon by limit cycles, and a crucial use of Dumortier's blow-up theorem
\cite{D77} allows one to restrict the attention to hyperbolic saddles.
More recent contributions in this direction are
by Marde\v{s}i\'c and collaborators \cite{MMV03, MMV08, MMSV15, MS07, Saav03}.

Our estimates only concern a single neutral saddle, and although for the purpose of Dulac's problem 
they can be treated by blow-ups, precise formulas for the Dulac times (and hence the Dulac map, see \eqref{eq:Dulac}), 
at cubic saddles in this generality seem to be new.

\subsection{Main results}\label{sec:main-results}

The crucial estimates here are of the Dulac times, i.e., the times that orbits take
to pass from an ``incoming'' unstable transversal to an ``outgoing'' unstable transversal 
to the flow, see Figure~\ref{fig:vf}.

\begin{theorem}\label{thm:perturb}
Consider a $C^3$ vector field of local form \eqref{eq:horivf} with parameters satisfying
\eqref{eq:para} and \eqref{eq:condL}.
Define
$$
\beta_0 := \frac{a_0+b_0}{2a_0}, \qquad  \beta_2 := \frac{a_2+b_2}{2b_2}, \qquad 
\beta_* = \frac{1}2\min\left\{ 1,\frac{a_2}{b_2}, \frac{b_0}{a_0}\right \}.
$$
Then there constants\footnote{The precise values of $\xi_0(\eta)$ and $\omega_0(\eta)$  are
given in in the proof Proposition~\ref{prop:mix_regvar}.} $\xi_0(\eta), \omega_0(\eta)$ 
such that the following asymptotics hold:
$$
\xi(\eta,\tilde T) = \xi_0(\eta) \tilde T^{-\beta_2} (1 + O(\tilde T^{-\beta_*}, T^{-\frac{1}{2}} \log T)).
$$ 
and
$$
\omega(\eta,\tilde T) = \omega_0(\eta) \tilde T^{-\beta_0} (1 + O(\tilde T^{-\beta_*}, T^{-\frac{1}{2}} \log T)).
$$ 
as $T \to \infty$.
\end{theorem}

In particular, the functions $\xi$ and $\omega$ are {\em regularly varying} of order $\beta_2$ in $T$, that is
$\lim_{T\to\infty} \frac{\xi(\eta,cT)}{\xi(\eta,T)} = c^{\beta_2}$ for every $c > 0$ and analogous for $\omega(\eta,T)$.
Moreover, the Dulac map  $D:W^u(0,\eta) \to W^s(\zeta_0,0)$ itself has the form (as $\xi \to 0$)
\begin{equation}\label{eq:Dulac}
\omega = D(\xi) = \omega_0(\eta) \xi_0(\eta)^{-\frac{\beta_0}{\beta_2} }
\, \xi^{\frac{\beta_0}{\beta_2} } \left( 1+\cO(\xi^{\frac{\beta_*}{\beta_2}},
-\xi^{\frac{1}{2\beta_2}} \log \xi ) \right).
\end{equation}

With assumptions \eqref{eq:divergencefree} and $c_1^2 < 4c_0 c_2$ in place, we can use the change of coordinates
$\bar x = \sqrt{a_0} x, \ \bar y = \sqrt{b_2}y$ and $\gamma = a_1/\sqrt{a_0 b_2} \in (-4,4)$ 
to transform \eqref{eq:polyvf} into the one-parameter family
\begin{equation}\label{eq:reducedvf}
\begin{pmatrix}
\dot{\bar x} \\ \dot{\bar y} \\ \dot{\bar x} 
\end{pmatrix}
 = 
\begin{pmatrix}
\bar x(\bar x^2+ \gamma \bar x \bar y + 3 \bar  y^2)) \\
-\bar y(3 \bar x^2 + \gamma \bar x \bar y + \bar y^2)) \\
1+\bar w(\bar x, \bar y)
\end{pmatrix} + \cO(4).
\end{equation}
for some transformed function $\bar w$.

Because of this genericity and reduced number of technicality that Lebesgue measure gives as opposed to
SRB-measure, we state our statistical result for volume preserving flows.
Theorem~\ref{thm:perturb} is used to estimate the measures of the strips $\{ \varphi = n\}$, see Figure~\ref{fig:leaves},
which in turn, together with the spectral properties of an induced Poincar\'e map $\hat f$ are 
crucial ingredients for the analysis required to establish the following stochastic limit properties
of the flow $\phi^t$.

\begin{corollary}\label{cor:LimitLaws}
 Consider a volume preserving almost Anosov flow \eqref{eq:reducedvf} on $\cM$ with $\gamma \in (-4,4)$
 and an observables $v:\cM \to \R$ that is $C^1$ on $\cM \setminus \Gamma$
 and has the form $v = v_0 + o(\rho)$ where $\int_0^{\tau} v_0 \circ \phi^t \, dt$ is 
 homogeneous of order $\rho \in (-2,1)$ in local coordinates $(x,y)$ near $p$ and $o(\rho)$ stands for terms of order $> \rho$.
 \begin{enumerate}
  \item If $\rho = 0$, then $v$  satisfies the Central Limit Theorem with non-standard scaling $\sqrt{t \log t}$, i.e.,
 $$
 \frac{\int_0^t v \circ \phi^s \, ds - t \int v \ dVol}{\sqrt{t \log t}} \Rightarrow_{dist} {\mathcal N}(0,\sigma^2)
 \quad \text{ as } t\to\infty,
 $$
 and the variance $\sigma^2 > 0$ unless $\int_0^\tau v \circ \phi^t \, dt$ is a coboundary.
 \item If $\rho > 0$, then $v$ satisfies the Gaussian Central Limit Theorem, i.e., with standard scaling $\sqrt{t}$.
 \item If $\rho \in (-2, 0)$ then $v$ satisfies a Stable Law of order $\frac{4}{2-\rho} \in (1,2)$.
 \end{enumerate}
\end{corollary}

Theorem~\ref{thm:perturb} allows also to derive other limit theorems such as in the 
infinite measure setting of \cite{BT17}, but with mixed terms.
But since we restrict to the Lebesgue measure (rather than SRB-measure) preserving case, we don't give any further details.

\subsection{Set-up}\label{sec:set-up}

The set-up here is largely taken over from \cite{BTT18}.
Our phase space will be the $3$-dimensional compact manifold $\cM$.

\begin{definition}{~\cite[Definition 1]{Hu00}}\label{def:aaf}
\label{def-AlmAn} 
A diffeomorphism $f:{\bT}^2\to {\bT}^2$ is called {\em almost Anosov} if there exists two continuous 
families of non-trivial cones $x\to\cC_x^u, \cC_x^s$ such that except for a finite set $S$,
\begin{itemize}
\item[i)] $D f_x\cC_{x}^u\subseteq \cC_{f(x)}^u$ and $D f_x\cC_{x}^s\supseteq \cC_{f(x)}^s$;
\item[ii)] $|D f_x v|>|v|$ for any $0 \neq v\in \cC_x^u$ and $|D f_x v|<|v|$ for any $0 \neq v\in \cC_x^s$.
\end{itemize}
For $x \in S$, $Df_x$ is the identity. 

A flow $f^t$ on $3$-torus $\T^3$ is called {\em almost Anosov flow} if it has a finite set $S$ of neutral 
periodic orbits, but everywhere else
observes the condition of an Anosov flow in that there is a continuous splitting of the
tangent bundle into a stable, an unstable and a neutral (flow) direction. For $x \in S$, the
derivative at the return time $\tau$ is $Df^\tau_x$ is the identity.
\end{definition}

The time-$1$ map $f$ of the flow $\phi^t$ of \eqref{eq:polyvf} 
has the form of a skew-product
\begin{equation}\label{eq:time1}
f\begin{pmatrix} x \\  y \\ z \end{pmatrix} =
\begin{pmatrix}
 x(1+a_0x^2+a_1 xy + a_2y^2) \\
 y(1-b_0x^2-b_1 xy - b_2y^2) \\
 z + O(|w(x,y)|)
\end{pmatrix} + \cO(3),
\end{equation}
see \cite[Section 2.1]{BT17}. Restricted to the $(x,y)$-coordinates, this map $f_{hor}$ is a smooth  
almost Anosov map with a single neutral fixed point $p = (0,0)$.
Let $\{ P_i \}_{i = 0}^k$ be the Markov partition for $f_{hor}$ (which we can assume to exist since $f_{hor}$ is a local
perturbation of a Anosov diffeomorphism on $\bT^2$).
We assume that $p$ belongs to the interior of $P_0$.
Clearly, the horizontal and vertical axes are the unstable and stable manifolds of $p$ respectively.
We assume that the Markov partition element $P_0 \subset U$ is a small rectangle such that 
$\overline{f_{hor}^{-1}(P_0) \cup P_0 \cup f_{hor}(P_0)} \subset U$.
Due to the symmetries $(x,y) \mapsto (\pm x, \pm y)$,
it suffices to do the analysis only in the first quadrant 
$Q = [0,\zeta_0] \times [0,\eta_0]$ of $P_0$, see Figure~\ref{fig:leaves}.
Without loss of generality (see \cite[Lemma 2.1]{BT17})
we can think of $[0, \zeta_0] \times \{ \eta_0 \}$ as a local unstable leaf
and $\{ \zeta_0\} \times [0, \eta_0]$ as a local stable leaf of the global diffeomorphism.

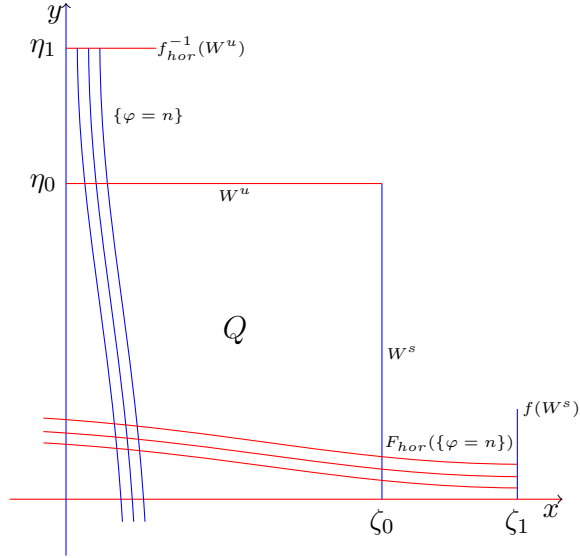
\begin{figure}[ht]
\begin{center}
\begin{tikzpicture}[scale=1.5]
\node at (4.3,-0.1) {\small $x$}; \node at (2.8,-0.2) {\small $\zeta_0$}; \node at (4,-0.2) {\small $\zeta_1$}; 
\node at (-0.1,4.3) {\small $y$}; \node at (-0.2,2.8) {\small $\eta_0$};\node at (-0.2,4) {\small $\eta_1$};
\draw[->, draw=red] (-0.5,0)--(4.4,0);
\draw[->, draw=blue] (0,-0.5)--(0,4.4);
\node at (0.73,3.4) {\tiny $\{ \varphi = n\}$};
\node at (3.4,0.5) {\tiny $F_{hor}(\{\varphi = n\})$};
\node at (1.5,1.5) {$Q$};
\node at (1.5,2.7) {\tiny $W^u$}; \node at (1.2,4) {\tiny $f_{hor}^{-1}(W^u)$};
\node at (2.98,1.3) {\tiny $W^s$}; \node at (4.3,0.8) {\tiny $f(W^s)$};
%
\draw[-, draw=red] (0,4)--(0.8,4); \draw[-, draw=blue] (4,0)--(4,0.8); 
\draw[-, draw=blue] (0.1,4) .. controls (0.11,2.5) and (0.4,1.5) .. (0.5,-0.2);
\draw[-, draw=blue] (0.2,4) .. controls (0.21,2.5) and (0.5,1.5) .. (0.6,-0.2);
\draw[-, draw=blue] (0.3,4) .. controls (0.31,2.5) and (0.6,1.5) .. (0.7,-0.2);
\draw[-, draw=blue] (2.8,2.8)--(2.8,-0.0);
\draw[-, draw=red] (4, 0.1) .. controls (2.5, 0.11) and (1.5, 0.4) .. (-0.2, 0.5);
\draw[-, draw=red] (4, 0.2) .. controls (2.5, 0.21) and (1.5, 0.5) .. (-0.2, 0.6);
\draw[-, draw=red] (4, 0.31) .. controls (2.5, 0.32) and (1.5, 0.61) .. (-0.2, 0.72);
\draw[-, draw=red] (2.8,2.8)--(-0.0,2.8);
\end{tikzpicture}
\caption{The first quadrant $Q$ of the rectangle $P_0$, with stable and unstabe foliations
drawn vertically and horizontally, respectively.}
\label{fig:leaves}
\end{center}
\end{figure} 

We consider an induced map $F_{hor} = f_{hor}^\varphi: Y \to Y$ for $Y := {\bT}^2 \setminus P_0$, where
$$
\varphi(z) = \min\{ n \geq 1 : f_{hor}^n(z) \notin P_0\}
$$
is the first return time to $Y$. Note that $F_{hor}$ is invertible because $f_{hor}$ is.
In the first quadrant of $U \setminus P_0$, $\{ \varphi = n \} := \{ z \in f^{-1}(Q) \setminus Q : \varphi(z) = n\}$, 
$n \geq 2$, are vertical strips 
adjacent to the local unstable leaf $[0,\zeta_0] \times \{\eta_0 \}$, and 
converging to $\{ 0 \} \times [\eta_0, \eta_1]$
as $n \to \infty$. The images
$F_{hor}(\{ \varphi = n\})$ are horizontal strips, adjacent to the local stable leaf $\{ \zeta_0 \} \times [0,\eta_0]$, 
and converging to $[\zeta_0, \zeta_1] \times \{ 0 \}$ as $n \to \infty$,
see Figure~\ref{fig:leaves}.

In contrast to $f_{hor}$, the induced map $F_{hor}$ is uniformly hyperbolic, but only piecewise continuous.
Indeed, continuity fails at the boundaries of the strips $\{ \varphi = n\}$, $n \geq 2$
(and $F$ is undefined on $W^s(p)$),
but these boundaries are local stable and unstable leaves,
and it is possible to create a countable Markov partition refining
$\{ P_i \}_{i=1}^k$ of $Y$ for $F$,
in which all the strips $\{ \varphi = n \}$ are partition elements.

\section{Regular variation of $\mu(\varphi > n)$ with mixed terms}\label{sec:regvarmixed}

In this section, we allow quadratic mixed terms in \eqref{eq:horivf}, but for the moment leave out the $\cO(4)$-terms.
That is, we consider 
\begin{equation}\label{eq:horivf2}
\begin{cases}
\dot x = x(a_0x^2+a_1xy+a_2y^2), \\
\dot y = -y(b_0x^2+b_1xy+b_2 y^2),
\end{cases} 
\end{equation}
that is, \eqref{eq:horivf} without the $\cO(4)$ terms but with the restrictions \eqref{eq:para} and \eqref{eq:condL}.
The condition $c_1^2 < 4c_0c_2$ avoids the formation of invariant lines $y = px$,
but in the below proofs it is used to guarantee that expressions as
$c_0 + c_1 M + c_2 M^2$ for $M = y/x$ are positive.
Our exposition closely follows \cite{BT17}, but  since the mixed terms require slight adjustments 
throughout the proof, we will give it in full.

Let $u,v \in \R$ be the solutions of the linear equations
\begin{equation}\label{eq:mix_uv}
\begin{cases}
(u+2) a_0 = v b_0 \\
(v+2) b_2 = u a_2
\end{cases}
\quad \text{ that is:}\quad
\begin{cases}
u = \frac{2b_2 c_0}{\Delta}, \\[2mm]
v = \frac{2a_0 c_2}{\Delta}.
\end{cases}
\end{equation}
Note that $u,v$ and $\Delta$ (recall $\Delta \neq 0$) all have the same sign and 
\eqref{eq:condL} implies that $\frac{b_1}{a_1} = \frac{u+1}{v+1}$. 
Compute that
\begin{equation}\label{eq:mix_gamma}
\beta_0 := \frac{a_0+b_0}{2a_0} = \frac{u+v+2}{2 v}, \quad 
\beta_2 := \frac{a_2+b_2}{2b_2} = \frac{u+v+2}{2 u}, \quad \frac{\beta_0}{\beta_2} = \frac{u}{v},
\end{equation}
and note that $\beta_0, \beta_2 > \frac{1}2$ (or $=\frac{1}2$ if 
we allow $b_0=0$ or $a_2=0$ respectively).
Under the extra assumption \eqref{eq:divergencefree} we obtain $\beta_0 = \beta_2 = 2$ and $u = v = 1$.

The first estimates is about the Dulac map of \eqref{eq:horivf}.

\begin{prop}\label{prop:mix_regvar}
Consider a vector field on the $2$-torus with local form \eqref{eq:horivf}
for $a_0, a_2, b_0, b_2 \geq 0$ and $\Delta \neq 0$.
There are functions $\xi_0(\eta), \omega_0(\eta), \xi_1(\eta), \omega_1(\eta) > 0$ 
independent of $T$ (with exact expressions given in the proof) 
such that 
$$
\xi(\eta, T) = \xi_0(\eta) T^{-\beta_2} \left(1 - \xi_1(\eta) T^{-1} + O(T^{-2}, T^{-2 \beta_2} )\right)
$$
and 
$$
\omega(\eta, T) = \omega_0(\eta) T^{-\beta_0} \left(1 - \omega_1(\eta) T^{-1} + O(T^{-2}, T^{-2 \beta_0} )\right).
$$
\end{prop}

\begin{lemma}\label{lem:mix_lf} The function
\begin{equation}\label{eq:mix_lf}
L(x,y) = 
\begin{cases}
  x^u y^v ( \frac{a_0}{v}\ x^2  + \frac{a_1}{v+1} xy + \frac{b_2}{u}\ y^2 ) & \text{ if } \Delta > 0,\\
  x^{-u} y^{-v} ( \frac{a_0}{v}\ x^2  + \frac{a_1}{v+1} xy \frac{b_2}{u}\ y^2 )^{-1} & \text{ if } \Delta < 0,
\end{cases}
\end{equation}
is a first integral of \eqref{eq:horivf}.
\end{lemma}

\begin{proofof}{Lemma~\ref{lem:mix_lf}}
First assume $\Delta > 0$, so $u,v > 0$ as well.
By \eqref{eq:mix_uv}, we can write $L(x,y)$ as
$$
L(x,y) = x^u y^v ( \frac{b_0}{u+2 }\ x^2  + \frac{b_2}{u}\ y^2 )
= x^u y^v ( \frac{a_0}{v}\ x^2  + \frac{a_2}{v+2 }\ y^2 ).
$$
Using these two equivalent expressions and that $\frac{a_1}{v+1} = \frac{b_1}{u+1}$. by
\eqref{eq:condL}, we compute the Lie derivative directly
\begin{eqnarray*}
\dot L &=& \langle \nabla L , X \rangle \\ &=&
x^{u-1}y^v\left(\frac{b_0}{u+2}(u+2) x^2 + \frac{a_1}{v+1}xy + \frac{b_2}{u} u x^{u-1} y^2\right)
x(a_0 x^2 + a_1 xy + a_2 y^2 )\\
&& - x^uy^{v-1}\left(\frac{a_0}{v} x^2 v + \frac{a_1}{v+1}xy+\frac{a_2}{v+2}(v+2)y^2\right)
y(b_0 x^2 + b_1 xy + b_2 y^2)\\
&=& 0.
\end{eqnarray*}
Any function of a first integral is a first integral, in particular this holds for $1/L$.
Therefore the conclusion is immediate for $\Delta < 0$ too.
\end{proofof}

\begin{proofof}{Proposition~\ref{prop:mix_regvar}}
We carry out the proof for $\Delta > 0$, so 
$L(x,y) =  x^u y^v ( \frac{a_0}{v}\ x^2  + \frac{a_1}{v+1} xy + \frac{b_2}{u}\ y^2 )$ as in Lemma~\ref{lem:mix_lf}.
The case $\Delta < 0$ goes likewise.
Fix $\eta$ such that $(\xi(\eta,T), \eta) \in \overline{\phi^{-1}(Q) \setminus Q}$.
For simplicity of notation, we will suppress the $\eta$ and $T$ in $\xi(\eta,T)$.
We use the variable $M = y/x$, so $y = Mx$ and differentiating gives
$\dot y = \dot M x + M \dot x$.
Recalling that $c_i = a_i+b_i$ and inserting the values for $\dot x$ and $\dot y$ from \eqref{eq:horivf},
we get
\begin{equation}\label{eq:mix_M0}
\dot M = -M( c_0 + c_1 M + c_2 M^2 ) x^2 .
\end{equation}
Assume that we are in the level set 
$L(x,y) =  L(\xi,\eta) = \xi^u \eta^v(\frac{a_0}{v} \xi^2  +\frac{a_1}{v+1} \xi\eta + \frac{b_2}{u}\eta^2 )$, then
we can solve for $x^2 $ in the expression
\begin{eqnarray*}
\xi^u\eta^v(\frac{a_0}{v} \xi^2  + \frac{a_1}{v+1} \xi\eta + \frac{b_2}{u}\eta^2 ) &=& 
x^u y^v (\frac{a_0}{v} x^2  + \frac{a_1}{v+1} xy + \frac{b_2}{u} y^2 ) \\
&=& x^{u+v+2 } M^v (\frac{a_0}{v} + \frac{a_1}{v+1} M + \frac{b_2}{u} M^2 ).
\end{eqnarray*}
Here we used
\begin{eqnarray}\label{eq:mix_rewrite}
\xi^u\eta^v(\frac{a_0}{v} \xi^2  + \frac{a_1}{v+1} \xi\eta + \frac{b_2}{u}\eta^2 ) &=& 
\xi^u\eta^v(\frac{a_0 \Delta}{2a_0 c_2} \xi^2  + \frac{a_1 \Delta}{2a_0c_2+\Delta} \xi\eta 
+ \frac{b_2 \Delta}{2b_2c_0}\eta^2 ) \nonumber \\
&=& \frac{\Delta}{2c_0c_2} \left(c_0 \xi^2 + \frac{2a_1 c_0c_2}{2a_0c_2+\Delta} \xi\eta 
+ c_2\eta^2\right) \nonumber \\
&=& \frac{\Delta}{2c_0c_2} \left(c_0 \xi^2 + c_1 \xi\eta 
+ c_2\eta^2\right),
\end{eqnarray}
(where the last step follows from \eqref{eq:condL}) and a similar computation for the term with $x,y$.

Use \eqref{eq:mix_uv} and \eqref{eq:mix_gamma} to obtain
$$
\begin{cases}
\frac{a_0}{v} + \frac{a_1}{v+1} M + \frac{b_2}{u} M^2  = \frac{\Delta}{2 c_0c_2} ( c_0+c_1M+ c_2M^2 ),
\\[1mm]
\frac{a_0 \xi^2}{v}  + \frac{a_1}{v+1} \xi\eta + \frac{b_2 \eta^2 }{u} 
= \frac{\Delta}{2 c_0c_2} ( c_0 \xi^2 + c_1 \xi \eta + c_2 \eta^2 ).
\end{cases}
$$
This gives 
\begin{eqnarray}\label{eq:x2}
x^2  &=& \xi^{\frac{2 u}{u+v+2 }}  \eta^{\frac{2 v}{u+v+2 }} M^{-\frac{2 v}{u+v+2 }}
\left(\frac{c_0\xi^2 + c_1\xi\eta + c_2\eta^2 }{c_0+c_1M+c_2M^2 } \right)^{\frac{2 }{u+v+2 }} \nonumber \\
 &=& \xi^{\frac{1}{\beta_2}}  \eta^{\frac{1}{\beta_0}} M^{-\frac{1}{\beta_0}}
 \left(\frac{c_0\xi^2 + c_1\xi\eta + c_2\eta^2 }{c_0+c_1M+c_2M^2 } \right)^{1 - \frac{1}{2 \beta_0} - \frac{1}{2 \beta_2}}
\end{eqnarray}
where recall $\beta_0 = \frac{u+v+2 }{2 v}$ and $\beta_2 = \frac{u+v+2 }{2 u}$ 
from \eqref{eq:mix_gamma}, which also gives $1-\frac{2 }{u+v+2 } = \frac{1}{2 \beta_0} + \frac{1}{2 \beta_2}$.
Combined with \eqref{eq:mix_M0}, this gives
\begin{equation}\label{eq:mix_M3}
 \dot M = -G M^{1-\frac{1}{\beta_0}} 
 \left(c_0+c_1M+c_2M^2  \right)^{\frac{1}{2 \beta_0} + \frac{1}{2 \beta_2}}
\end{equation}
with
\begin{equation}\label{eq:mix_Exi}
G = G(\xi,\eta) := \xi^{ \frac{1}{\beta_2} } \eta^{ \frac{1}{\beta_0} }
\left( c_0\xi^2  + c_1\xi\eta + c_2\eta^2  \right)^{1-\frac{1}{2 \beta_0}-\frac{1}{2 \beta_2}}.
\end{equation}
For the exit time $T \geq 0$, recall that $\xi(\eta,T)$ and $\omega(\eta,T)$ are such that 
the solution of \eqref{eq:horivf} satisfies
$(x(0), y(0)) = (\xi(\eta,T), \eta)$ and $(x(T), y(T)) = (\zeta_0,\omega(\eta,T))$.
This implies $M(0) = \eta/\xi(\eta,T)$ and $M(T) = \omega(\eta,T)/\zeta_0$.
Inserting this in \eqref{eq:mix_M3}, separating variables, and integrating we get
\begin{equation}\label{eq:mix_M4}
\int_{\omega(\eta,T)/\zeta_0}^{\eta/\xi(\eta,T)} \frac{ M^{\frac{1}{\beta_0}-1}  \, dM}
{  \left(c_0+c_1M+c_2 M^2  \right)^{ \frac{1}{2 \beta_0} + \frac{1}{2 \beta_2} } }
= G(\xi(\eta,T), \eta) T.
\end{equation}
In the rest of the proof, we will frequently suppress the dependence on $\eta$ and $T$ in $\xi(\eta,T)$
and $\omega(\eta,T)$. We know that 
$L(\xi(\eta, T), \eta) = \xi^u \eta^v(\frac{a_0}{v} \xi^2  + \frac{b_2}{u}\eta^2 ) = 
\zeta_0^u \omega^v (\frac{a_0}{v} \zeta_0^2 + \frac{b_2}{u} \omega^2 )
= L(\eta,\omega(\eta,T))$,
which gives
\begin{equation}\label{eq:mix_xieta}
\xi^u \eta^v (c_0\xi^2 +c_1\xi \eta+ c_2\eta^2 ) = \zeta_0^u \omega^v(c_0\zeta_0^2+c_1\zeta_0\omega+c_2\omega^2 ).
\end{equation}
From their definition, $\xi(\eta,T)$ and $\omega(\eta, T)$ are clearly decreasing in $T$,
so their $T$-derivatives $\xi'(\eta,T), \omega'(\eta,T) \leq 0$. 
Since $c_0, c_2 > 0$ (otherwise $\Delta = 0$), the integrand of \eqref{eq:mix_M4} is $O(M^{\frac{1}{\beta_0}-1})$ 
as $M \to 0$ and $O(M^{-\frac{1}{\beta_2}-1})$ as $M \to \infty$. Hence the integral is increasing and bounded in $T$.
But this means that $G(\xi(\eta,T), \eta) T$ is increasing in $T$ and bounded as well.
Let $g(\eta,T) = \xi(\eta,T) T^{\beta_2}$. Since
$$
G(\xi(\eta,T), \eta) T  = g(\eta,T)^{\frac{1}{\beta_2}} \eta^{\frac{1}{\beta_0}} (c_0 g(\eta,T)^2  T^{-2  \beta_2}
+ c_1 g(\eta,T) T^{-\beta_2} + c_2 \eta^2 )^{1-\frac{1}{2 \beta_0} - \frac{1}{2 \beta_2}},
$$
and $1-\frac{1}{2 \beta_0} - \frac{1}{2 \beta_2} > 0$,
we find that $g(\eta,T)$ 
converges\footnote{For the symmetric statement on $\omega(\eta,T)$, define $\hat g(\eta,T) = \omega(\eta, T) T^{\beta_0}$.
Then $\lim_{T \to \infty} \hat g(\eta,T) = \lim_{T \to \infty} g(\eta,T)^{ \beta_0/\beta_2 } \eta^{1+2 /v} 
\zeta_0^{-b_0/a_0} (\frac{c_2}{c_0})^{1/v}$.}:
\begin{equation}\label{eq:mix_xi0}
\xi_0(\eta) := \lim_{T\to\infty} g(\eta,T) = c_2^{-\frac1u} \eta^{- \frac{a_2}{b_2} } \left( \int_0^\infty 
\frac{ M^{\frac{1}{\beta_0}-1} \, dM }
{\left( c_0 + c_1 M + c_2 M^2  \right)^{ \frac{1}{2 \beta_0} + \frac{1}{2 \beta_2} } }\right)^{\beta_2} ,
\end{equation}
where we have used $-\beta_2(1-\frac{1}{2 \beta_0} - \frac{1}{2 \beta_2}) = -\frac{2 \beta_2}{u+v+2 } = -\frac{1}{u}$
for the exponent of $c_2$, and $\frac{2 }{u} + \frac{\beta_2}{\beta_0} = \frac{v+2 }{u} = \frac{a_2}{b_2}$ 
for the exponent of $\eta$.

We continue the proof to get higher asymptotics.
Differentiating \eqref{eq:mix_M4} w.r.t.\ $T$ gives
\begin{equation}\label{eq:mix_Mprime}
-\frac{ \eta^{\frac{1}{\beta_0}} \xi^{\frac{1}{\beta_2}-1} \xi'}{(c_0 \xi^2  + c_1\xi\eta + c_2\eta^2 )^{\frac{1}{2 \beta_0} + \frac{1}{2 \beta_2}}}
-\frac{\zeta_0^{\frac{1}{\beta_2}} \omega^{\frac{1}{\beta_0}-1} \omega'}{(c_0\zeta_0^2 + c_1 \zeta_0\omega + c_2 \omega^2 )^{\frac{1}{2 \beta_0} + \frac{1}{2 \beta_2}}}
=  \frac{\partial G(\xi,\eta)}{\partial\xi} T \xi'  + G(\xi,\eta),
\end{equation}
where (by differentiating \eqref{eq:mix_Exi})
$$
\frac{\partial G(\xi,\eta)}{\partial\xi}  = 
(2b_0\xi^2 + c_1(\frac{b_0}{c_0}+\frac{a_2}{c_2})\xi\eta + b_2\eta^2 ) \xi^{\frac{1}{\beta_2}-1} \eta^{\frac{1}{\beta_0}}
 (c_0\xi^2  + c_1 \xi\eta + c_2 \eta^2 )^{-\frac{1}{2 \beta_0} - \frac{1}{2 \beta_2}}.
$$
Combined with \eqref{eq:mix_Exi}, \eqref{eq:mix_xieta} and \eqref{eq:mix_Mprime}, this gives
\begin{align}\label{eq:mix_xi2}
-\eta^{\frac{1}{\beta_0}} \xi' &- \zeta_0^{\frac{1}{\beta_2}}
\left(\frac{\zeta_0^u \omega^v}{\xi^u \eta^v} \right)^{\frac{1}{2 \beta_0} + \frac{1}{2 \beta_2}} 
\frac{\omega^{\frac{1}{\beta_0}-1}}{\xi^{\frac{1}{\beta_2}-1}}  \omega' \nonumber \\
&=(2b_0\xi^2 + c_1(\frac{b_0}{c_0}+\frac{a_2}{c_2})\xi\eta + b_2\eta^2 )  T \eta^{\frac{1}{\beta_0}} \xi' 
+ \eta^{\frac{1}{\beta_0}} \xi (c_0 \xi^2 + c_1 \xi\eta + c_2 \eta^2 ).
\end{align}
Because $\frac{1}{2 \beta_0} + \frac{1}{2 \beta_2} - 1 = -\frac{2 }{u+v+2 }$, using
\eqref{eq:mix_gamma} and dividing by $\eta^{\frac{1}{\beta_0}}$, we can simplify \eqref{eq:mix_xi2} to 
\begin{equation}\label{eq:mix_xi2a}
- \xi' - \frac{\zeta_0^u}{\eta^v} \frac{\omega^{v-1}}{\xi^{u-1}} \omega' =
(2b_0\xi^2 + c_1(\frac{b_0}{c_0}+\frac{a_2}{c_2})\xi\eta + b_2\eta^2 ) 
T \xi' + \xi (c_0 \xi^2 + c_1 \xi\eta + c_2 \eta^2 ).
\end{equation}
Taking the derivative of \eqref{eq:mix_xieta} w.r.t.\ $T$  and multiplying with $\Delta/(c_0c_2)$ gives
$$
(2b_0\xi^2 + c_1(\frac{b_0}{c_0}+\frac{a_2}{c_2})\xi\eta + b_2\eta^2 )  \eta^v \xi^{u-1} \xi' = 
(2a_0\zeta_0^2 + + c_1(\frac{b_0}{c_0}+\frac{a_2}{c_2})\zeta_0\omega+ 2a_2 \omega^2 )\zeta_0^u \omega^{v-1} \omega'.
$$
Hence, we can rewrite \eqref{eq:mix_xi2a} as
\begin{align*}
-\Big(1 +  &\frac{2b_0\xi^2 + c_1(\frac{b_0}{c_0}+\frac{a_2}{c_2})\xi\eta + 2b_2\eta^2}
{2a_0\zeta_0^2 + c_1(\frac{b_0}{c_0}+\frac{a_2}{c_2})\zeta_0\omega + 2a_2\omega^2 }\Big) \xi' \\
& =\
(2b_0\xi^2 + c_1(\frac{b_0}{c_0}+\frac{a_2}{c_2})\xi\eta + b_2\eta^2 )  T  \xi' +  \xi (c_0 \xi^2 +
c_1 \xi\eta + c_2 \eta^2 ).
\end{align*}
We insert $\xi' = g'(T) T^{-\beta_2} - \beta_2 g(T) T^{-(1+\beta_2)}$ and multiply with $T^{\beta_2}$,
which leads to
\begin{align*}
-\Big(1+ &\frac{2b_0\xi^2 + c_1(\frac{b_0}{c_0}+\frac{a_2}{c_2})\xi\eta + 2b_2\eta^2}
{2a_0\zeta_0^2 + c_1(\frac{b_0}{c_0}+\frac{a_2}{c_2})\zeta_0\omega + 2a_2\omega^2 }\Big) (g'(T) - \beta_2 g(T) T^{-1})
\\
&=\ (2b_0\xi^2  + c_1 (\frac{b_0}{c_0}+\frac{a_2}{c_2})\xi\eta + 2b_2\eta^2) g'(T) \ T 
-\frac{\Delta}{b_2} g(T)^3 T^{-2 \beta_2}.
\end{align*}
Since $\xi = O(T^{-\beta_2})$ and $\omega = O(T^{-\beta_0})$, we can write this differential equation as
$$
\frac{g'}{g} = \frac1{T^2} \frac{\beta_2}2 
\frac{\frac{a_0\zeta_0^2 +b_2 \eta^2 + O(T^{-2 \beta_2})}{a_0\zeta_0^2 +O(T^{-2 \beta_0})}
-\frac{\Delta}{b_2} g(T)^2 T^{-\frac{a_2}{b_2}} }
{b_2 \eta^2 + O(T^{-2 \beta_2}) + O(T^{-1})}.
$$
Keeping the leading
terms only (where we use that $2\beta_2, 2\beta_0 > 1$), we get the differential equation
$$
\frac{g'}{g} = (\xi_1(\eta) + O(\max\{T^{-1}, T^{-\frac{a_2}{b_2}}\} ) ) \frac{1}{T^2} \quad \text{ for }
\xi_1 = \xi_1(\eta) := \frac{\beta_2}2 \left( \frac{1}{a_0\zeta_0^2}+ \frac{1}{b_2\eta^2 } \right).
$$
Using the limit boundary value $\xi_0 = \xi_0(\eta) = \lim_{T \to \infty} g(\eta,T)$, we find the solution
$$
g(\eta,T) = \xi_0 e^{-(\xi_1 + O(\max\{T^{-1}, T^{-\frac{a_2}{b_2}}\} ) )T^{-1} } 
= \xi_0( 1 - \xi_1 T^{-1}  + O(\max\{T^{-2}, T^{-2 \beta_2}\} ) )
$$
as required. The analogous asymptotics for $\omega$ and the constants $\omega_0$ and $\omega_1$
can be derived by changing the time direction and the roles 
$(a_0, a_2) \leftrightarrow (b_2, b_0)$, and also by the relation $\xi^u \eta^{v+2 }c_2 \sim \zeta_0^{u+2} \omega^v c_0$
from \eqref{eq:mix_xieta}:
$$
\omega_0(\eta) := c_0^{-\frac{1}{v}} \zeta_0^{-\frac{b_0}{a_0}} \left( \int_0^\infty \frac{ M^{\frac{1}{\beta_2}-1}  \, dM }
{\left( c_0 M^2  + c_1 M + c_2 \right)^{ \frac{1}{2 \beta_0} + \frac{1}{2 \beta_2} } }
\right)^{\beta_0}
\ \text{ and } \ \omega_1(\eta) := \frac{\beta_0}{2 } \left(\frac{1}{b_2\zeta_0^2}+ \frac{1}{a_0\eta^2 }\right).
$$
This concludes the proof.
\end{proofof}

\section{Proof of Theorem~\ref{thm:perturb}}\label{sec:perturb}

To prove that the regular variation established in Proposition~\ref{prop:mix_regvar}
is robust under perturbations of the vector field, we put the $\cO(4)$ terms back into \eqref{eq:horivf},
but since we consider it as a perturbation of \eqref{eq:horivf2},, we write $\tilde X$ instead:

\begin{equation}\label{eq:perturb}
\tilde X = \begin{pmatrix} \tilde X_1 \\ \tilde X_2 \end{pmatrix}
= \begin{pmatrix}
   x(a_0 x^2 + a_1xy + a_2 y^2) \\
  -y(b_0 x^2 + b_1xy + b_2 y^2)
  \end{pmatrix} + \cO(4),
\end{equation}
so that $|\tilde X - X| = \cO(4)$.
The quantities $\xi(\eta, T), \omega(\eta,T)$ will be written as $\tilde \xi(\eta, T), \tilde\omega(\eta,T)$ etc., 
and the goal is to show that $\tilde \xi(\eta, T)$ is still regularly varying.
Let us now give the proof of Theorem~\ref{thm:perturb}.

\begin{proof}
As before, let $\xi = \xi( \eta,T)$ be such that for the unperturbed flow, 
$\phi^T (\xi, \eta) = (\zeta_0, \omega(\eta,T))$.
Proposition~\ref{prop:mix_regvar} gives the asymptotics of $\xi(\eta,T)$ as $T \to \infty$.
At the same time, under the perturbed flow associated to \eqref{eq:perturb},
$\phi^{\tilde T} (\xi, \eta) = (\zeta_0, \tilde \omega(\eta,\tilde T))$ for some $\tilde T$.
Therefore we can write $\xi(\eta,T)  = \tilde \xi(\eta, \tilde T)$,
and once we estimated $\tilde T$ as function of $T$, we can express $\tilde \xi(\eta, \tilde T)$
explicitly as function of $\tilde T$. We follow the argument of the proof of Proposition~\ref{prop:mix_regvar}, 
keeping track of the effect of the higher order terms. \\[3mm]
{\bf The perturbed first integral:} 
To start, we construct a first integral 
 $\tilde L$ on $Q = [0,\zeta_0] \times [0,\eta_0]$ by defining 
 $$
 \tilde L(\tilde \phi^t(\delta, \delta)) = L(\delta,\delta) = 
 \begin{cases}
  \delta^{u+v+2}(\frac{a_0}{v} + \frac{a_1}{v+1} + \frac{b_2}{u})  & \text{ if } \Delta > 0,\\
  \delta^{-(u+v+2)}(\frac{a_0}{v} + \frac{a_1}{v+1} + \frac{b_2}{u})^{-1}  & \text{ if } \Delta < 0,
 \end{cases}
 $$ 
 for $0 < \delta \leq \min\{\zeta_0,  \eta_0\}$ and $t \in \R$.
 (We continue the argument for the case $\Delta > 0$; the other case goes analogously.)
 
 By construction, $\tilde L$ is constant on integral curves
 of $\dot z = \tilde X(z)$.
 Because $\tilde X$ is $C^3$, the integral curves are $C^3$ curves,
 and form a $C^3$ foliation of $P_0$, see e.g.\ \cite[Theorem 2.10]{Teschl}.
 Note that the coordinate axes consist
 of the stationary point $(0,0)$ and its stable and unstable manifold;
 we put $\tilde L(x,0) = \tilde L(0,y) = 0$.
 Then $\tilde L$ is continuous on $Q$ and $C^{2+1}$ on the interior of $Q$.
 
 Now we compare $\tilde L$ with $L$ on a small neighbourhood $U$ of 
 $\phi_{hor}^{-1}(Q) \cup Q \cup \phi_{hor}^1(Q)$.
 Take $y_0 = \eta_0$ and $x_0 = x_0(\delta)$ such that the integral curve
 of $\dot z = X(z)$ through $z_0 := (x_0, y_0)$  intersects the diagonal at $(\delta, \delta)$.
 Then the integral curve
 of $\dot z = \tilde X(z)$ through $z_0$ intersects the diagonal at $(\tilde \delta, \tilde \delta)$
 for some $\tilde \delta = \tilde \delta(\delta)$, see Figure~\ref{fig:deltas}.
\begin{figure}[ht]
\begin{center}
\begin{tikzpicture}[scale=1.3]
\node at (3.6,-0.15) {\small $x$};\node at (-0.16,3.6) {\small $y$};
\draw[->] (-0.5,0)--(3.6,0);
\draw[->] (0,-0.5)--(0,3.6);
\draw[-] (0,0)--(2,2);
\node at (0.16, 2.8) {\small $\bullet$}; \node at (0.34,2.8) {\small $z$};
\node at (0.525, 0.525) {\small $\bullet$}; \node at (-0.1,0.48) {\small $\delta$};
\node at (0.68, 0.68) {\small $\bullet$}; \node at (-0.1,0.74) {\small $\tilde\delta$};
\draw[-] (1,-0.05)--(1,0.05); \node at (1,-0.25) {\small $\tilde x(\delta)$};
\draw[->, draw=blue] (0.16,2.8) .. controls (0.2,0.3) and (0.5,0.5) .. (1.3,0.25);
\draw[->, draw=red] (0.16,2.8) .. controls (0.25,0.3) and (0.7,0.7) .. (1.35,0.4);
\node at (9.6,-0.15) {\small $x$};\node at (5.84,3.6) {\small $y$};
\draw[->] (5.5,0)--(9.6,0);
\draw[->] (6,-0.5)--(6,3.6);
\draw[-] (6,0)--(8,2);
\node at (6.16, 2.8) {\small $\bullet$}; \node at (6.35,2.8) {\small $z$};
\node at (6.54, 0.54) {\small $\bullet$}; \node at (5.9,0.63) {\small $\delta$};
\node at (6.43, 0.43) {\small $\bullet$}; \node at (5.9,0.4) {\small $\tilde\delta$};
\draw[-] (6.3,-0.05)--(6.3,0.05); \node at (6.34,-0.25) {\small $\tilde x(\delta)$};
\draw[->, draw=blue] (6.16,2.8) .. controls (6.2,0.3) and (6.5,0.5) .. (7.3,0.25);
\draw[->, draw=red] (6.16,2.8) .. controls (6.15,0.3) and (6.3,0.3) .. (7.35,0.1);
\end{tikzpicture}
\caption{Solutions of \eqref{eq:vfq1} and \eqref{eq:vfq2}, starting from the same point
$z = (x, y)$. The left and right panel refer to the cases $\tilde\delta > \delta$
and $\tilde \delta < \delta$ respectively.}
\label{fig:deltas}
\end{center}
\end{figure}
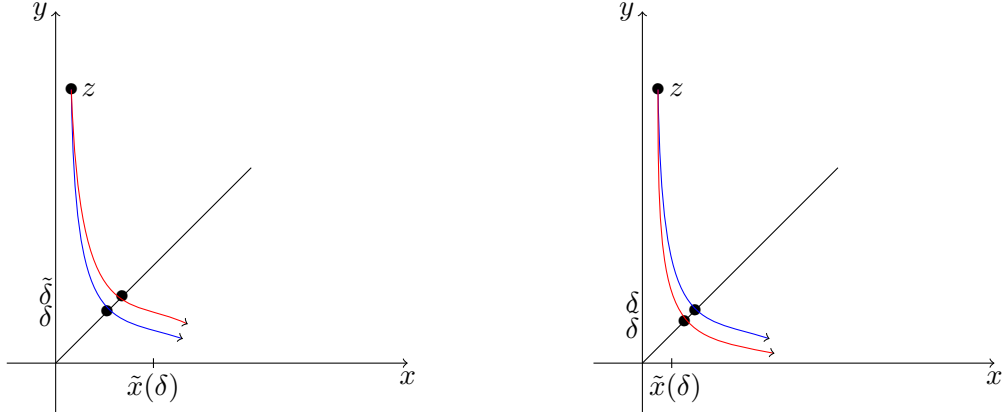 

Therefore
\begin{equation}\label{eq:LL}
\tilde L(z) \ = \ \tilde L(\tilde \delta, \tilde\delta)\ =\  L(\delta,\delta)
\left(\frac{\tilde\delta}{\delta}\right)^{u+v+2}
\ =  \ L(z) \left( \frac{\tilde\delta}{\delta}\right)^{u+v+2} \!\!\!\!\!\!\! \!\!\!\!\!\!\!\!\!\!.
\end{equation}
\noindent
{\bf Estimating $\tilde \delta/\delta$:} Parametrise the integral curve of $X$ through $z_0$ as
$(x(y), y)$ for $\min\{ \delta, \tilde \delta\} \leq y \leq y_0$.
(So $x \leq y$; the case $y \leq x$ can be dealt with by switching the roles of $x$ and $y$.)
Then by \eqref{eq:horivf}:
\begin{equation}\label{eq:vfq1}
 x'(y) = -\frac{x(a_0 x^2 + a_1xy+ a_2 y^2)}{y(b_0 x^2 + b_1 xy + b_2 y^2)}.
\end{equation}
For the perturbed vector field \eqref{eq:perturb} we parametrise the integral curve of through $z_0$ as
$(\tilde x(y), y)$ and we have the analogue of \eqref{eq:vfq1}:
\begin{equation}\label{eq:vfq2}
 \tilde x'(y) = -\frac{\tilde x(a_0 \tilde x^2 + a_1 \tilde x y + a_2 y^2 
 + \sum_{j=0}^3 \hat a_j \tilde x^j y^{3-j} + o(|(\tilde x,y)|^3))}
 {y(b_0 \tilde x^2 + b_1 \tilde x y+  b_2 y^2 + \sum_{j=0}^3 \hat b_j \tilde x^j y^{3-j}
 + o(|(\tilde x,y)|^3))}.
\end{equation}
Since $x \leq y$, the $O$-terms can be written as $O(y^3)$.
Combining \eqref{eq:vfq1} and \eqref{eq:vfq2} we obtain
\begin{eqnarray*}
\tilde x'(y)  &=&  -\frac{\tilde x(a_0 \tilde x^2 + a_1 \tilde x y + a_2 y^2)}
{y(b_0\tilde x^2 + b_1 \tilde x y + b_2 y^2)}
(1+ q_0 \tilde x + q_2 y +  o(|(\tilde x,y)|)) \\
&=& x'(y)(1+q_0 x + q_2 y +  o(|(\tilde x,y)|)).
\end{eqnarray*}
We will neglect the term $o(|(\tilde x,y)|)$ because they can be absorbed
in the big-$O$ terms at the end of the estimate.
Integration over $[\delta,y_0]$ gives
$$
\tilde x(y_0) - \tilde x(\delta) = x(y_0) - x(\delta)  + q_0 \int_\delta^{y_0} x'(y) x(y)\ dy
+ q_2 \int_\delta^{y_0} x'(y) y\ dy.
$$
Since $\tilde x(y_0) = x(y_0) = x_0$ and $x(\delta) = \delta$, this simplifies to
\begin{eqnarray}\label{eq:dtilded}
\tilde x(\delta) - \delta &=&  -  \frac{q_0}2 \int_\delta^{y_0} (x^2(y))' \ dy -
q_2 \int_\delta^{y_0} (x(y) y)' \ dy + q_2 \int_\delta^{y_0} x(y) \ dy  \nonumber \\
&=& \frac{q_0}2(\delta^2 - x_0^2) + q_2 \left( \delta^2 - x_0 y_0 + \int_\delta^{y_0} x(y) \ dy \right).
\end{eqnarray}
We solve for $x$ from 
$x^u y^v (\frac{a_0}{v} x^2 + \frac{a_1}{v+1} xy + \frac{b_2}{u}y^2) = L(x, y) = L(\delta, \delta) 
= \delta^{u+v+2}(\frac{a_0}{v}+\frac{a_1}{v+1}+\frac{b_2}{u})$:
\begin{eqnarray}\label{eq:x(y)}
 x = x(y) &=& \delta^{\frac{u+v+2}{u}} y^{-\frac{v+2}{u}} (1+\frac{a_1 u}{b_1(v+1)} + \frac{u a_0}{v b_2})^{\frac1u}
 (1+\frac{a_1 u}{b_1 (v+1)} + \frac{c_0}{c_2} \frac{x^2}{y^2})^{-\frac1u} \nonumber \\
  &=& \underbrace{(c_2+\frac{a_1 u c_2}{b_1 (v+1)}+c_0)^{\frac1u} 
  (c_2+\frac{a_1 u C_2}{b_1 (v+1)} + c_0 \frac{x^2}{y^2})^{-\frac1u}}_{U(y)}   \delta^{1+\frac{a_2}{b_2}} y^{-\frac{a_2}{b_2}}.
\end{eqnarray}
In particular,
$$
x_0 = x(y_0) = \delta^{1+\frac{a_2}{b_2}}(1++\frac{a_1 u}{b_1 (v+1)}+\frac{c_0}{c_2})^{\frac1u} y_0^{-\frac{a_2}{b_2}} 
  (1+O(\delta^{2(1+\frac{a_2}{b_2})})). 
$$
Combine the first two factors of \eqref{eq:x(y)} to
$$
U(y) :=  (c_2+\frac{a_1 u c_2}{b_1 (v+1)}+c_0)^{\frac{1}{u}} (c_2+\frac{a_1 uc_2}{b_1 (v+1)} + c_0 \frac{x^2}{y^2})^{-\frac1u} 
\in [1, (1+\frac{c_0}{c_2})^{\frac{1}{u}}].
$$
Note that $\lim_{y \to \delta} U(y) = 1$, and $U(y)$ is differentiable.
Using \eqref{eq:vfq1} and \eqref{eq:x(y)} we compute the derivative
$$
U'(y) = \frac{\Delta c_0}{b_2}\ \frac{U(y)}{b_2+b_0 (\frac{x}{y})^2}\ \frac1y\ (\frac{x}{y})^2
= \frac{\Delta c_0}{b_2}\ \delta^{2(1+\frac{a_2}{b_2})}\  \frac{U(y)^3}{b_2+b_0 (\frac{x}{y})^2} 
\ y^{-2(1+\frac{a_2}{b_2})-1}.
$$
Next we integrate by parts (assuming first that $\frac{a_2}{b_2} \neq 1$):
\begin{eqnarray*}
\int_\delta^{y_0} x(y) \ dy &=& \delta^{1+\frac{a_2}{b_2}} \int_\delta^{y_0} U(y) y^{-\frac{a_2}{b_2}} dy
= \frac{b_2}{b_2-a_2}
\left( U(y_0) \delta^{1+\frac{a_2}{b_2}} y_0^{1-\frac{a_2}{b_2}} - \delta^2 \right) \\
&&  \qquad \underbrace{- \frac{\Delta c_0}{b_2-a_2}\ \delta^{3(1+\frac{a_2}{b_2})}\  
\int_\delta^{y_0} \frac{U(y)^3}{b_2+b_0 (\frac{x}{y})^2} \  
y^{1-3(1+\frac{a_2}{b_2})} dy}_I.
\end{eqnarray*}
Since $\frac{1}{b_2+b_0} \leq  \frac{U(y)^3}{b_2+b_0 (\frac{x}{y})^2}
\leq \frac{1}{b_2}(1+\frac{c_0}{c_2})^{\frac{3}{u}}$
and $\frac{U(y)^3}{b_2+b_0 (\frac{x}{y})^2} \to \frac{1}{b_2+b_0}$ as $y \to \delta$,
there are constants $\hat C_1, \hat C_2 \in \R$ 
such that the final term in the above expression is
$$
I = \hat C_1 \delta^2 + \hat C_2 \delta^{3(1+\frac{a_2}{b_2})}
y_0^{2-3(1+\frac{a_2}{b_2})} + O(\delta^3).
$$
For the case $\frac{a_2}{b_2} = 1$, a similar computation gives
$$
\int_\delta^{y_0} x(y) \ dy = \hat C_3 \delta^2 \log \delta + \hat C_4 \delta^2 \log y_0
+ \hat C_5 \delta^6 y_0^{-4} \log y_0
+  \hat C_6 \delta^6 y_0^{-4} + O(\delta^3 \log \delta),
$$
for some generically nonzero $\hat C_3, \hat C_4, \hat C_5, \hat C_6 \in \R$.

By \eqref{eq:vfq2}, the derivative $\tilde x'(\delta) = \frac{a_0+a_1+a_2}{b_0+b_1+b_2} + O(\delta)$.
Since $\tilde \delta$ lies between $\delta$ and $\tilde x(\delta)$ (see Figure~\ref{fig:deltas}), 
we have 
\begin{eqnarray}\label{eq:xdelta}
|\tilde x(\delta) - \delta| &=& |\tilde x(\delta) - \tilde \delta| + |\tilde \delta - \delta| \nonumber \\
&=& \left(1+ \frac{a_0+a_1+a_2}{b_0+b_1+b_2} + O(\delta)\right) |\tilde \delta - \delta|
= \frac{c_0+c_1+c_2 + O(\delta) }{b_0+b_1+b_2}|\tilde \delta - \delta|.
\end{eqnarray}
Later in the proof we need the quantity
$$
\psi(\delta) := \left(\frac{\tilde\delta}{\delta}\right)^{u+v+2} - 1 = 
(u+v+2) \frac{\tilde \delta - \delta}{\delta} + O\left(\frac{|\tilde \delta - \delta|^2}{\delta^2}\right).
$$
Writing $|\tilde \delta-\delta|$ in terms of $|\tilde x(\delta)-\delta|$ using
\eqref{eq:xdelta}, and
combining with the above estimates for $|\tilde x(\delta)-\delta|$, we find 
\begin{eqnarray}\label{eq:psidelta}
\psi(\delta) &=& C_1 \delta + C_2 \delta^{\frac{a_2}{b_2}} y_0^{1-\frac{a_2}{b_2}} 
 + C_3 \delta^{1+\frac{a_2}{b_2}} y_0^{-\frac{a_2}{b_2}} 
 + C_4 \delta^{1+\frac{2a_2}{b_2}} y_0^{-\frac{2a_2}{b_2}} \nonumber  \\
&& +\ C_{\log} \delta \log \delta + O(\delta^2, \delta^{\frac{2a_2}{b_2}} )
\end{eqnarray}
for (generically nonzero) constants $C_1, C_2, C_3, C_4 \in \R$
and $ C_{\log}$ is only nonzero if $\frac{a_2}{b_2} = 1$.

For the region $\{ x \geq y \}$ (containing the point $(x_1, y_1) := (\zeta_0, \tilde \omega(\eta,\tilde T))$)
we reverse the roles  $a_2,b_2,x_0, y_0 \leftrightarrow b_0, a_0, y_1, x_1$.
This gives
\begin{eqnarray*}
\psi(\delta) &=& \hat C_1 \delta + \hat C_2 \delta^{\frac{b_0}{a_0}} x_1^{1-\frac{b_0}{a_0}} 
 + \hat C_3 \delta^{1+\frac{b_0}{a_0}} x_1^{-\frac{b_0}{a_0}} 
 + \hat C_4 \delta^{1+\frac{2b_0}{a_0}} x_1^{-\frac{2b_0}{a_0}} \nonumber  \\
&& +\ \hat C_{\log} \delta \log \delta + O(\delta^2, \delta^{\frac{2b_0}{a_0}} ),
\end{eqnarray*}
for (generically nonzero) constants $\hat C_1, \hat C_2, \hat C_3, \hat C_4 \in \R$
and $\hat C_{\log}$ is only nonzero if $\frac{b_0}{a_0} = 1$.
Combining with \eqref{eq:psidelta} gives
\begin{equation}\label{eq:psidelta2}
 \psi(\delta) =
\begin{cases}
O(\delta \log 1/\delta) & \text{ if }  \min\{\frac{a_2}{b_2}, \frac{b_0}{a_0}\} = 1,\\
 O(\delta^{a_*}) & \text{ otherwise, with } a_* = \min\{ 1, \frac{a_2}{b_2}, \frac{b_0}{a_0}\}.
\end{cases}
\end{equation}

\noindent
{\bf Estimate of $\tilde T$:} Now let $z_0 = (x_0, y_0) = (\xi(\eta, T), \eta) = (\tilde\xi(\eta, \tilde T), \eta)$ 
be the point such that $\phi_{hor}^T(z_0) = (\zeta_0, \omega(\eta,T))$ under the {\bf unperturbed}
flow and $\tilde \phi_{hor}^{\tilde T}(z_0) = (\zeta_0, \tilde \omega(\eta,\tilde T))$ under the {\bf perturbed}
flow. We estimate $\tilde T$ in terms of $T$.

Combining the estimate for $\xi(\eta, T)$ from Proposition~\ref{prop:mix_regvar}
with $L(\delta,\delta) = L(\xi(\eta,T), \eta)$, we can find the relation
between $\delta$ and $T$:
\begin{equation}\label{eq:deltaT}
 \delta = \delta_0 T^{-\frac{1}2} (1 +\frac{\xi_1}{2\beta_2} T^{-1}
 + O(T^{-2}, T^{-2 \beta_2})),
\end{equation}
for $\delta_0 = \xi_0^{\frac{1}{2 \beta_2}} \eta^{1-\frac{1}{2 \beta_2}}
(\frac{c_2}{c_0+c_2})^{\frac{1}{u+v+2}}$.

For $M = y/x$, computations analogous to \eqref{eq:mix_M0} show that there is
$\Psi = \Psi(x,M) = O(1+M^3)$ such that
$$
\dot M = - M (c_0 + c_1M + c_2 M^2 + x \Psi) x^2.
$$
For every $(x,y) = (x,xM)$ on the $\tilde \phi$-trajectory of $z_0$ (i.e., level set of $\tilde L$),
we have
$$
\xi^u \eta^v (\frac{a_0}{v} \xi^2 + \frac{a_1}{v+1} \xi \eta + \frac{b_2}{u} \eta^2)(1+\psi(\xi,\eta)) = 
x^{u+v+2} M^2 (\frac{a_0}{v} + \frac{a_1}{v+1} M + \frac{b_2}{u} M^2) (1+\psi(x,xM)).
$$
This gives the analogue of \cite[formula (32)]{BT17}
\begin{equation}\label{eq:M3a}
 \dot M = - G(\xi,\eta) M^{1-\frac{1}{\beta_0}}
 \left(c_0+c_1M + c_2M^2 + x\Psi(x,M) \right)^{\frac{1}{2 \beta_0} + \frac{1}{2 \beta_2}}
\left( \frac{ 1+\psi(x,y) }{1+\psi(\xi,\eta) } \right)^{1-\frac{1}{2 \beta_0} - \frac{1}{2 \beta_2}},
\end{equation}
where $G(\xi,\eta)$ is as in \eqref{eq:mix_Exi}. 
To estimate $\tilde T$, we take some increasing function 
$\delta \leq \rho(\delta) \leq \delta^{1/2}$ such that $\delta = o(\rho(\delta))$ and divide the 
trajectory $\tilde \phi^t(z_0) = (\tilde x(t), \tilde y(t))$  of $z_0$ into three parts
separated by two points in time:
\begin{equation}\label{eq:T1T2}
\tilde T_1 = \min\{ t > 0 : \tilde y(t) = \rho(\delta) \},
\qquad
\tilde T_2 = \max\{ t < \tilde T :  \tilde x(t) = \rho(\delta) \},
\end{equation}
and let $T_1, T_2$ be the analogous quantities for the unperturbed trajectory.
We compute
$$
T_1 = \int_{y_0}^{\rho(\delta)} \frac{dy}{\dot y} = 
\int_{y_0}^{\rho(\delta)} \frac{dy}{-y(b_0 x(y)^2 + b_1 \tilde x(y)y + b_2 y^2)} 
= O(\rho(\delta)^{-2}).
$$
Similarly, using $\tilde x(y)/x(y) = 1+O(\psi(\delta))$ as in \eqref{eq:x(y)},
\begin{eqnarray*}
\tilde T_1 - T_1 &=& 
\int_{y_0}^{\rho(\delta)}\frac{1+O(y)}{-y(b_0 \tilde x(y)^2 + b_1\tilde x(y)y + b_2 y^2)} 
-  \frac{1}{-y(b_0 x(y)^2 + b_2 y^2)} \ dy \\
&=& \int_{y_0}^{\rho(\delta)}\frac{O(y)( 1+O(\psi(\delta)))}{-y(b_0 x(y)^2 + b_1 \tilde x(y)y + b_2 y^2)} \ dy 
= O(\rho(\delta)^{1-2}).
\end{eqnarray*}
This gives
\begin{equation}\label{eq:tildeT1}
\tilde T_1 = T_1(1+ O(T_1^{-\frac12})) \quad \text{ and } \quad
\tilde T - \tilde T_2 = (T-T_2) (1+O(\rho( (T-T_2)^{-\frac12}) ))
\end{equation}
by a similar computation for $T-T_2 = \int_{\rho(\delta)}^{\zeta_0}\frac{dx}{\dot x}$, etc.

Finally, for $\tilde T_1 < t < \tilde T_2$, we have $\psi(x, y) = O(\delta^{\alpha_*}, \delta \log(1/\delta))$ 
by \eqref{eq:psidelta2}, and $x\Psi(x,M) = O(x + y M^2) = (1+M^2) O(\rho(\delta))$.
Therefore 
\begin{eqnarray*}
\tilde T_2 - \tilde T_1 &=& \int_{M(\tilde T_2)}^{M(\tilde T_1)} 
\frac{  M^{\frac{1}{\beta_0}-1}  \left(1+\frac{ O(x+yM^2) }{c_0+c_1M+c_2 M^2} \right)^{\frac{1}{2 \beta_0} + \frac{1}{2 \beta_2}} }
{ G(\xi,\eta)(c_0+c_1M+c_2 M^2)^{\frac{1}{2 \beta_0} + \frac{1}{2 \beta_2}} }
\left( \frac{ 1+\psi(x,y) }{1+\psi(\xi,\eta) } \right)^{\frac{1}{2 \beta_0} + \frac{1}{2 \beta_2}-1}\ dM\\[1mm]
&=& \int_{M(\tilde T_2)}^{M(\tilde T_1)} 
\frac{  M^{\frac{1}{\beta_0}-1} \left( 1+ O(\rho(\delta))  \right) }
{ G(\xi,\eta)(c_0+c_1M+c_2 M^2)^{\frac{1}{2 \beta_0} + \frac{1}{2 \beta_2}} }
\left(1+O(\delta^{\alpha_*}, \delta \log(1/\delta))\right)  \ dM\\[2mm]
&=& (T_2-T_1)(1+O(\rho(\delta), \delta^{\alpha_*}, \delta \log(1/\delta)).
\end{eqnarray*}
Choosing $\rho(\delta) = \delta \log(1/\delta)$,
and using \eqref{eq:deltaT} gives
$$
\tilde T_2 - \tilde T_1 = 
(T_2-T_1)(1+O(T^{-\frac{\alpha_*}2} , T^{-\frac{1}2} \log T) ).
$$
Combining this with \eqref{eq:tildeT1} 
gives
$\tilde T = T(1+O(T^{-\beta_*}, T^{-\frac{1}2}  \log T))$ for 
$\beta_* =  \frac{1}2 \min\{1, \frac{a_2}{b_2}, \frac{b_0}{a_0} \}$.
The estimate of Proposition~\ref{prop:mix_regvar} now gives 
$\tilde \xi(\eta, \tilde T) = 
\xi_0(\eta) \tilde T^{-\beta_2}(1+O(\tilde T^{-\beta_*}, \tilde T^{-\frac{1}2} \log \tilde T))$ 
as claimed.

Reversing the roles $(a_0, a_2) \leftrightarrow (b_2, b_0)$ as in the end of the proof of 
Proposition~\ref{prop:mix_regvar} gives
$\tilde \omega(\eta, \tilde T) = 
\omega_0(\eta) \tilde T^{-\beta_0}(1+O(\tilde T^{-\beta_*}, \tilde T^{-\frac{1}2} \log \tilde T))$. 
\end{proof}

The formula \eqref{eq:Dulac} for the Dulac maps follows directly from Theorem~\ref{thm:perturb} by inverting 
$T \mapsto \xi(\eta,T)$ and inserting this in the formula for $\omega(\eta,T)$.
In the special case that $\beta_0=\beta_2$, formula \eqref{eq:Dulac} reduces to
$$
\omega = D(\xi) = \left(\frac{b_2}{a_0}\right)^{\frac{1}{\beta_2-1}}
\left( \frac{\eta}{\zeta_0} \right)^{2\beta_2-1} \xi \left(1+\cO(\xi^{1-\frac{1}{2\beta_2}}, -\xi^{\frac{1}{2\beta_2}}
\log \xi)\right).
$$
Reducing further by assuming \eqref{eq:divergencefree} (i.e., in the volume preserving setting), we get 
$$
\omega = D(\xi) = \frac{b_2}{a_0}\left( \frac{\eta}{\zeta_0} \right)^3 \xi \left(1+\cO(-\xi^{\frac14}\log \xi)\right).
$$
This coefficient 
$\frac{b_2}{a_0}\left( \frac{\eta}{\zeta_0} \right)^3 = \frac{\| X_{hor}(0,\eta) \|}{\| X_{hor}(\zeta_0,0) \|}$ 
agrees with the fact that for $\omega = D(\xi)$, the flow-boxes $\cup_{t \in [0,\eps]} \phi_{hor}^t([0,\xi] \times \{\eta\})$
and $\cup_{t \in [0,\eps]} \phi_{hor}^t(\{\zeta_0\} \times [0,\omega])$
must have the same volume. If the neutral saddle $p$ is part of a heteroclinic cycle, then it is accumulated by 
periodic solutions, but these are not limit cycles of course.

\section{Time-$1$ map versus Poincar\'e map}\label{sec:Poincare}

First we give an estimate of observables integrated over the flow-lines of $X_{hor}$ of \eqref{eq:horivf}.

\begin{prop}\label{prop:r-integral}
 Let $r = \sqrt{x^2+y^2}$, $\rho > 0$ and $W(T)$ be the integral curve for \eqref{eq:horivf}
 connecting $(\xi(\eta_0, T), \eta_0))$ to $(\zeta_0, \omega(\eta_0, T)$,
 see Figure~\ref{fig:vf}. 
 Then there is a constant $C = C(\rho) > 0$ such that 
\begin{equation}\label{eq:r-integral}
\Theta := \int_{W(T)} r(t)^\rho dt =
\begin{cases}
C T^{1-\frac{\rho}{2}} (1+o(1)) & \text{ if } \rho < 2, \\
C \log(T) (1+o(1)) & \text{ if } \rho = 2, \\
C (1+o(1)) & \text{ if } \rho > 2.
\end{cases}
\end{equation}
\end{prop}

\begin{proof}
 We build on the proof of Proposition~\ref{prop:mix_regvar} (or in fact Theorem~\ref{thm:perturb}), and in the integral
 $\Psi$ we change coordinates $M = y/x$.
 That is, $r^\rho = (x^2+y^2)^{\rho/2} = x^{\rho} (1+M^2)^{\rho/2}$.
 Use \eqref{eq:x2} to get
 $$
 x = G(T)^{\frac{1}{2}} M^{-\frac{v}{u+v+2 }} (c_0+c_1M+c_2M^2)^{\frac{-1}{u+v+2 }},
 $$
 with $G(T) := G(\xi(\eta_0, T)), \eta_0)$ as in \eqref{eq:mix_Exi}.
 Abbreviate $\xi(\eta_0, T) = \xi(T)$ and $\omega(\eta_0, T) = \omega(T)$.
Inserting the above in the integral of \eqref{eq:mix_M4}, we obtain
\begin{equation}\label{eq:Psi-integral}
\Psi = G(T)^{\frac{\rho}{2}-1} 
\int_{\omega(T)/\zeta_0}^{\eta/\xi(T)} 
\frac{ M^{-\frac{\rho v}{u+v+2}} (c_0+c_1M+c_2M^2)^{\frac{-\rho}{u+v+2 }} (1+M^2)^{\frac{\rho}{2}}}
{  M^{1-\frac{1}{\beta_0}}  \left(c_0+c_1M+c_2 M^2  \right)^{ \frac{1}{2\beta_0} + \frac{1}{2\beta_2} } }
\ dM.
\end{equation}
For $M \to 0$, the leading term in the integrand is
$$
c_0^{-1+(1-\frac{\rho}{2}) \frac{1}{\beta_0}} 
M^{\frac{1}{\beta_0}-1 - \rho \frac{v}{u+v+2} } = 
c_0^{-1+(1-\frac{\rho}{2}) \frac{1}{\beta_0}}  M^{(1-\frac{\rho}{2})\frac{1}{\beta_0} - 1},
$$
i.e., the exponent is $> -1$ for $\rho < 2$.
For $M \to \infty$, the leading term in the integrand is
$$
c_2^{-1+(1-\frac{\rho}{2}) \frac{1}{\beta_0}}  
M^{-\frac{1}{\beta_2}-1 - \rho(\frac{v}{u+v+2} + \frac{2}{u+v+2} -1) } = 
c_0^{-1+(1-\frac{\rho}{2}) \frac{1}{\beta_0}}  M^{-(1-\frac{\rho}{2})\frac{1}{\beta_2} - 1},
$$
i.e., the exponent is $< -1$ for $\rho < 2$.
This means that the integral in \eqref{eq:Psi-integral} converges to some constant $C_0 = C_0(\rho)$
as $T \to \infty$,
and $\Psi \sim C_0  G(T)^{\frac{\rho}{2}-1}  \sim C T^{1-\frac{\rho}{2}}$
for $C = C_0 \left( c_2^{1-\frac{1}{2 \beta_0} - \frac{1}{2 \beta_2}} 
\xi_0^{\frac{1}{\beta_0}} \eta_0^{1-\frac{1}{\beta_2}} \right)^{1-\frac{\rho}{2}}$.
This finishes the proof for $\rho < 2$.

If $\rho > 2$, then the value of $\Psi$ based on the leading terms of the integrand only,
is
$$
\Psi = G(T)^{\frac{\rho}{2}-1} 
\frac{c_0^{-1+(1-\frac{\rho}{2}) \frac{1}{\beta_0}} } {\frac{\rho}{2}-1} 
\left( \beta_2 
\left( \frac{\eta_0}{\xi(T)}\right)^{-(1-\frac{\rho}{2})\frac{1}{\beta_2}}
- \beta_0
\left(\frac{\omega(T)}{\zeta_0} \right)^{(1-\frac{\rho}{2})\frac{1}{\beta_0}} \right).
$$
Insert the values of $\xi(T)$ and $\omega(T)$ from Proposition~\ref{prop:mix_regvar}
as well as the leading term of $G(T)$:
$$
\Psi = \left( c_2^{1-\frac{1}{2 \beta_0} - \frac{1}{2 \beta_2}} 
\xi_0^{\frac{1}{\beta_0}} \eta_0^{1-\frac{1}{\beta_2}} \right)^{1-\frac{\rho}{2}} 
T^{1-\frac{\rho}{2}}
\frac{c_0^{-1+(1-\frac{\rho}{2}) \frac{1}{\beta_0}} } {\frac{\rho}{2}-1} 
\left( \beta_2 {\eta_0}^{(\frac{\rho}{2}-1)\frac{1}{\beta_2}}
- \beta_0 \zeta_0^{(\frac{\rho}{2}-1)\frac{1}{\beta_0}} \right) T^{\frac{\rho}{2}-1}.
$$
The powers of $T$ cancel in this expression, proving the case $\rho > 2$.
Finally, if $\rho = 2$, then the factor $G(T)^{\frac{\rho}{2}-1}$ in \eqref{eq:Psi-integral}
disappears and the leading terms in the integrand (both as $M \to 0$ and $M \to \infty$),
are $c_0^{-1} M^{-1}$.
This gives, due to Proposition~\ref{prop:mix_regvar},
$$
\Psi \sim \frac{1}{c_0} \left( \log \frac{\eta_0}{\xi(T)} - \log \frac{\omega(T)}{\zeta_0} \right)
\sim \frac{\beta_2+\beta_0}{c_0} \log T.
$$
\end{proof}

The $3$-dimensional time-$1$ map $\phi^1$ preserves no $2$-dimensional submanifold of $\cM$. 
Yet in order to model $\phi^t$ as a suspension flow over a $2$-dimensional map, we need a genuine
Poincar\'e map.
For this we choose a section $\Sigma$ transversal to $\Gamma$ and containing a neighbourhood $U$ of $p$.
As an example, $\Sigma$ could be $\T^2 \times \{ 0 \}$, and the Poincar\'e map to $\T^2 \times \{ 0 \}$ 
could be (a local perturbation of) Arnol'd's cat map; in this case (and most cases)
$\cM$ is not homeomorphic to $\T^3$ because the homology is more complicated, see \cite{BF13, N76}.

Let $h:\Sigma \to \R^+$, $h(q) = \min\{ t > 0 : \phi^t(q)  \in \Sigma\}$ be the first return time.
Assuming that $\sup_\Sigma |w(x,y)| < 1$, the
first return time $h$ is bounded and bounded away from zero, say
$0 < \inf_{\Sigma} h < \sup_{\Sigma} h$.

The Poincar\'e map $f := \phi^h: \Sigma \to \Sigma$ 
has a neutral saddle point $p$ at the origin. Its local stable/unstable manifolds
are $W^s_{loc}(p) = \{ 0 \} \times (-\eps,\eps)$ and 
$W^u_{loc}(p) = (-\eps,\eps) \times \{ 0 \}$.
Because the flow $\phi^t$ is a perturbation of an Anosov flow, and $f$ is 
a Poincar\'e map, it has a finite Markov partition $\{P_i\}_{i \geq 0}$
and we can assume that $p$ is in the interior of $P_0$.
In the sequel, let $U$ be a neighbourhood of $p$ that is small enough that 
\eqref{eq:polyvf} is valid on $U \times [0,1]$ but also that $f(U) \supset \hat P_0 \cup P_0$.

In order to regain the hyperbolicity lacking in $f$, let 
\begin{equation}\label{eq:firstreturn}
\rf(q) := \min\{ n \geq 1 : f^n(q) \in Y \}
\end{equation}
be the first return time to $Y := \Sigma \setminus P_0$.
Then the Poincar\'e map $F = f^{\rf} = \phi^\tau$ of $\phi^t$ to
$Y \times \{ 0 \}$ is hyperbolic, where
\begin{equation*}
\tau(q) = \min\{ t > 0 : \phi^t((q,0)) \in Y \times \{ 0 \} \} = \sum_{j=0}^{\rf-1} h \circ f^j 
\end{equation*}
is the corresponding first return time. 

Consequently, the flow $\phi^t:\cM \times \R \to \cM$ can be modeled as a suspension flow
on $Y^\tau = \left( \bigcup_{q \in Y} \{ q \} \times [0,\tau(q)) \right)/(q,\tau(q)) \sim (F(q),0)$.
Since the flow and section $Y \times \{ 0 \}$ are $C^1$ smooth,
$\tau$ is $C^1$ on each piece $\{\rf = k\}$.

\begin{lemma}\label{lem:tau}
In the notation of Proposition~\ref{prop:r-integral} with $\theta = w$, we have
$\tau(q) = \hat\tau(q) + O(1)$ and $\rf = \hat\tau(q)+\Theta(\hat\tau(q))+O(1)$.
\end{lemma}

\begin{proof}
By the definition of $\hat\tau$ we have $\phi^{\hat\tau}_{hor}(q) \in \hat W^s$. Therefore it takes a bounded amount of
 time (positive or negative) for $\phi^{\hat\tau}(q,0)$ to hit $Y \times \{ 0 \}$, so $|\tau(q)-\hat\tau(q)| = O(1)$.
 
If in \eqref{eq:r-integral} we set $\theta = w$, then $\hat\tau(q) + \Theta(\hat\tau(q))$ indicates the vertical 
displacement under the flow $\phi^t$. In particular, it gives the number of times the flow-line intersects 
$\Sigma$, and hence $\rf = \hat\tau(q) + \Theta(\hat\tau(q)) + O(1)$. 
\end{proof}

\begin{figure}[ht]
\begin{center}
\begin{tikzpicture}[scale=1.5]
\node at (4.3,-0.1) {\small $x$}; 
\node at (-0.1,4.3) {\small $y$}; \node at (-0.2,2.8) {\small $y_1$};\node at (-0.2,4) {\small $y_2$};
\draw[->, draw=red] (-0.5,0)--(4.4,0);
\draw[->, draw=blue] (0,-0.5)--(0,4.4);
\node at (0.73,3.4) {\tiny $\{ \rf = k\}$};
\node at (3.4,0.5) {\tiny $F(\{\rf = k\})$};
\node at (1.5,1.5) {$P_0$};
\node at (1.5,2.7) {\tiny $W^u$}; \node at (1.2,4) {\tiny $f^{-1}(W^u)$};
\node at (2.98,1.3) {\tiny $W^s$}; \node at (4.3,0.8) {\tiny $f(W^s)$};
\draw[-, draw=red] (0,4)--(0.8,4); \draw[-, draw=blue] (4,0)--(4,0.8); 
\draw[-, draw=blue] (0.1,4) .. controls (0.11,2.5) and (0.4,1.5) .. (0.5,-0.2);
\draw[-, draw=blue] (0.2,4) .. controls (0.21,2.5) and (0.5,1.5) .. (0.6,-0.2);
\draw[-, draw=blue] (0.3,4) .. controls (0.31,2.5) and (0.6,1.5) .. (0.7,-0.2);
\draw[-, draw=blue] (2.8,2.8)--(2.8,-0.0);
\draw[-, draw=red] (4, 0.1) .. controls (2.5, 0.11) and (1.5, 0.4) .. (-0.2, 0.5);
\draw[-, draw=red] (4, 0.2) .. controls (2.5, 0.21) and (1.5, 0.5) .. (-0.2, 0.6);
\draw[-, draw=red] (4, 0.31) .. controls (2.5, 0.32) and (1.5, 0.61) .. (-0.2, 0.72);
\draw[-, draw=red] (2.8,2.8)--(-0.0,2.8);
\draw[->, draw=black] (0.085,4) .. controls (0.27,0.5) and (0.5,0.27) .. (4,0.085);
\end{tikzpicture}
\caption{The first quadrant of the rectangle $P_0$, with stable and unstable foliations of 
Poincar\'e map $f = \phi^h$
drawn vertically and horizontally, respectively. Also one of the integral curves is drawn.}
\label{fig:leaves2}
\end{center}
\end{figure}
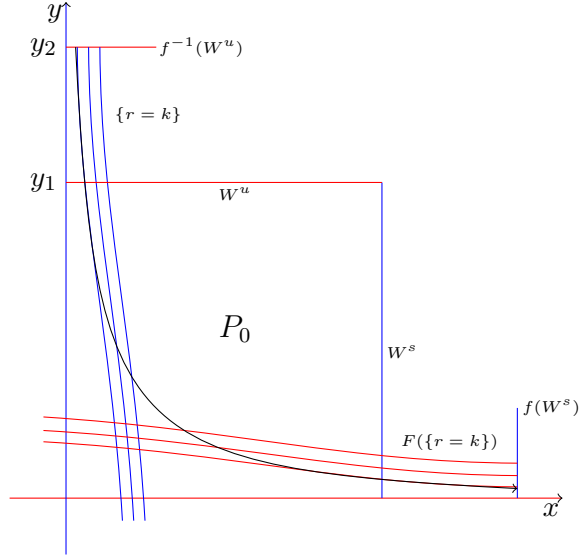 

Assume that $\phi^t$ and $\hat f$ preserve Lebesgue measure.

\begin{prop}\label{prop:tailtau}
Recall that $\beta_2 = \frac{a_2+b_2}{2b_2} \in (\frac12, \infty)$.
There exists $C^* > 0$ such that
\begin{equation}\label{eq:asymp3}
\Leb(\{ \tau > t \}) = C^* t^{-\beta_2} (1+o(1))
\end{equation}
for the $F$-invariant SRB-measure $\mu_{\bar\phi}$.
\end{prop}

\begin{proof}
The function $\tau$ is defined on $\Sigma \setminus P_0$ and $\tau \geq h_2 = h + h \circ f$ on
$Y_{\{\rf \geq 2\}} := f^{-1}(P_0) \setminus P_0$.
The set $Y_{\{\rf \geq 2\}}$ is a rectangle with boundaries consisting of two 
stable and two unstable leaves of the Poincar\'e map  $f$.
Let $W^u(y)$ denote the unstable leaf of $f$ inside $Y_{\{\rf \geq 2\}}$
with $(0,y)$ as (left) boundary point.
Let $y_1 < y_2$ be such that $W^u(y_1)$ and $W^u(y_2)$ are the unstable
boundary leaves of $Y_{\{\rf \geq 2\}}$.

The unstable foliation of $\hat f = \phi_{hor}^1$ does not entirely coincide with the unstable foliation of $f$.
Let $\hat W^u(y)$ denote the unstable leaf of $\hat f$ 
with $(0,y)$ as (left) boundary point. Both $\hat W^u(y)$ and $W^u(y)$ are $C^1$ curves emanating from $(0,y)$; 
let $\gamma(y)$ denote the angle between them. Then the lengths
\begin{eqnarray*}
\Leb(W^u(y) \cap \{ \tau > t\}) &=& |\cos \gamma(y)|\ \Leb(\hat W^u(y) \cap \{ \tau > t\}) (1+o(1)) \\
&=& |\cos \gamma(y)|\ \xi_0(y)\ t^{-\beta_2} (1+o(1))
\end{eqnarray*}
as $t\to\infty$, where the last equality and the notation $\xi_0(y)$ and  $\beta_2 = (a_2+b_2)/(2b_2)$ come
from Theorem~\ref{thm:perturb}

We decompose Lebesgue on $Y_{\{\rf \geq 2\}}$ as 
$$
\int_{Y_{\{\rf \geq 2\}}} v\, d\mu_{\bar\phi} = 
\int_{y_1}^{y_2} \left( \int_{W^u(y) } v\, d\mu^s_{W^u(y)} \right) d\nu^u(y).
$$
The conditional measures $\mu_{W^u(y)}$ on $W^u(y)$
equals $1$-dimensional Lebesgue $m_{W^u(y)}$ on $W^u(y)$
Therefore, as $t \to \infty$,
\begin{eqnarray*}
 \mu_{\bar\phi}(\tau > t) &=& \int_{y_1}^{y_2} \mu_{W^u(y)}(W^u(y) \cap \{ \tau > t\} ) \, d\nu^u(y) \\
  &=& \int_{y_1}^{y_2} m_{W^u(y)}( W^u(y) \cap \{ \tau > t\} ) \, d\nu^u(y) \\
   &=& \int_{y_1}^{y_2} |\cos \gamma(y)|\ m_{\hat W^u(y)}(\hat W^u(y) \cap \{ \tau > t\} ) (1+o(1)) \, d\nu^u(y) \\
    &=& \int_{y_1}^{y_2} |\cos \gamma(y)|\ \xi_0(y)\ t^{-\beta_2} (1+o(1)) \, d\nu^u(y) 
    = C^* t^{-\beta_2}(1+o(1)),
\end{eqnarray*}
for $C^* = \int_{y_1}^{y_2} |\cos \gamma(y)| \ \xi_0(y) \, d\nu^u(y)$.
This proves the result.
\end{proof}

\section{The proof of Corollary~\ref{cor:LimitLaws} }
\label{sec:proof}

\begin{proof}
The proof of Corollary~\ref{cor:LimitLaws} is a direct application of Theorem 2.7 in \cite{BTT18},
where $\bar v = \int_0^\tau v \circ \phi^t\, dt$ takes the role of $\bar \psi$ in
\cite[Theorem 2.7]{BTT18}, but the condition that $\bar \psi = C-\psi_0$ for some positive $\psi_0$ is only
important for the results on the shape of the pressure function in \cite{BTT18}.
For us, only the tail of $\bar v$ matters and since $v$ is $C^1$ on $\cM \setminus \Gamma$,
$\bar v$ is $C^1$ on each partition element $\{ \phi = n\}$ of the Markov map $F$.
Since Proposition~\ref{prop:r-integral} applies to $v$ we get
$\bar v(x,y) \sim C_p T^{1-\frac{\rho}{2}}$ if the Dulac time of $(x,y)$ is $T$.
Since our invariant measure is Lebesgue, and $\beta_2 = 2$,
Theorem~\ref{thm:perturb} can be immediately used to estimate
$$
\Leb(\bar v > t) \sim \frac{\int_{\eta_0}^{\eta_1} \xi_0(\eta) \, d\eta}{\mbox{Vol }(\Sigma \setminus P_0)} 
\left( \frac{t}{C_p} \right)^{\frac{-4}{2-\rho } },
$$
where $\Sigma$ is the Poincar\'e section and $\mbox{Vol }(\Sigma \setminus P_0)$ is the normalizing 
constant for Lebesgue restricted to the domain $\Sigma \setminus P_0$ of $F$.
If $\rho \geq 2$, this asymptotic formula should be interpreted as $\Leb(\bar v > t) = 0$ for $t$ large, 
that is: $\bar v$ is bounded.

The exponent of this tail is $-2$ if and only if $\rho = 0$, and in this case
\cite[Theorem 2.7(a)(ii)]{BTT18} gives the non-Gaussian CLT.

If $-2 < \rho < 0$, \cite[Theorem 2.7(a)(i)]{BTT18} gives a Stable Law of order $4/(2-\rho) \in (1,2)$.

Finally, if $0 < \rho < 2$ (or $\rho \geq 2$ when $\bar v$ is bounded), 
then we obtain the CLT provided the variance $\sigma^2 > 0$, 
and this follows from $\bar v$ not being a coboundary.
In other words, $\bar v \neq h - h \circ F$ for any $h \in \cB$, the Banach space used in the proofs
of \cite{BTT18}, and this we assumed explicitly.
\end{proof}

\noindent
{\bf Acknowledgements:} We gratefully acknowledge the support of FWF grant P31950-N45.

\end{document}